\documentclass[12pt]{amsart}
\usepackage{amsthm,amsmath,xspace,url,gensymb,amsfonts}
\usepackage{graphicx,overpic}
\usepackage{hyperref}
\usepackage{enumerate}

\title[Cyclic Sieving for torsion pairs of type $A_n$]{Cyclic Sieving
  for torsion pairs in the cluster category of Dynkin type $A_n$}

\author{Stefan Kluge}
\email{kluge.ish@web.de}

\author{Martin Rubey}
\address{Institut f\"ur Algebra, Zahlentheorie und Diskrete Mathematik,Leibniz
  Universit\"at Hannover, Welfengarten 1, D-30167 Hannover, Germany}
\email{martin.rubey@math.uni-hannover.de}
\urladdr{http://www.iazd.uni-hannover.de/~rubey/}

\keywords{cluster category, torsion pair, Ptolemy diagram, polygon
  dissection, cyclic sieving}

\newtheorem{thm}{Theorem}[section]
\newtheorem{lem}[thm]{Lemma}
\newtheorem{cor}[thm]{Corollary}
\newtheorem{prop}[thm]{Proposition}

\theoremstyle{definition}
\newtheorem{dfn}[thm]{Definition}
 
\theoremstyle{remark}
\newtheorem{rmk}{Remark}

\newtheorem*{eg*}{Example}

\newcommand{\Set}[1]{\ensuremath{\mathcal{#1}}}  
\newcommand{\Diag}[1]{\ensuremath{\mathfrak{#1}}}
\newcommand{\Cat}[1]{\ensuremath{\mathsf{#1}}}   
\newcommand{\Dfn}[1]{\emph{#1}}                  
\newcommand{\Hom}{Hom}

\newcommand{\integers}{\mathbb Z}
\newcommand{\size}[1]{\left\lvert #1\right\rvert}

\newcommand{\qi}[2][q]{[#2]_{#1}}
\newcommand{\qbinom}[3][q]{\genfrac{[}{]}{0pt}{}{#2}{#3}_{#1}}

\newcommand{\mymod}[1]{\!\!\!\pmod{#1}}

\DeclareMathOperator{\nc}{nc}

\def\yv{z}
\def\tv{x}
\def\pv{y_1}
\def\cv{y_2}
\def\yk{\bar z}
\def\tk{\bar x}
\def\xck{\bar y_1}
\def\xdk{\bar y_2}
\def\tr{k}
\def\po{\ell}
\def\cl{m}
\def\di{d}
\def\bi{b}

\begin{document}
\maketitle

\begin{abstract} 
  Recently, a combinatorial model for torsion pairs in the cluster
  category of Dynkin type $A_n$ was introduced, and used to derive an
  explicit formula for their number.  In this article we determine
  the number of torsion pairs that are invariant under $\bi$-fold
  application of Auslander-Reiten translation.

  It turns out that the set of torsion pairs together with
  Auslander-Reiten translation, and a natural $q$-analogue of the
  formula for the number of all torsion pairs exhibits the cyclic
  sieving phenomenon.
\end{abstract}
\section{Introduction}
\label{sec:introduction}

\subsection{Torsion Pairs and Ptolemy Diagrams}
Very recently, Ptolemy diagrams were introduced by Thorsten Holm,
Peter J\o rgensen and Martin Rubey in~\cite{HolmJorgensenRubey2010}
as a combinatorial model for torsion pairs in the cluster category of
Dynkin type $A_n$.  Similar to triangulations of the $(n+3)$-gon,
which can be regarded as a combinatorial model for tilting objects,
Ptolemy diagrams are certain subsets of the set of (proper) diagonals
of an $(n+3)$-gon with a distinguished base edge.  As in the case of
triangulations, each such diagonal corresponds to an indecomposable
object in the cluster category.

\begin{figure}[h]
  \centering
  \raisebox{16pt}{$\Set P=$}%
  \begin{overpic}[clip=true,trim=60 80 460 710,angle=270]{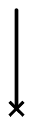}
  \end{overpic}
  \raisebox{16pt}{$\mathaccent\cdot\cup$}%
  \begin{overpic}[clip=true,trim=60 80 460 710,angle=270]{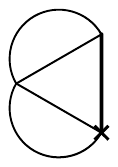}
    \put(14,35){\Set P}
    \put(51,35){\Set P}
  \end{overpic}
  \raisebox{16pt}{$\mathaccent\cdot\cup$}%
  \begin{overpic}[clip=true,trim=60 70 460 700,angle=270]{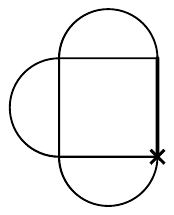}
    \put(17,35){\Set P}
    \put(72,35){\Set P}
    \put(45,65){\Set P}
  \end{overpic}
  \raisebox{16pt}{$\mathaccent\cdot\cup$}%
  \begin{overpic}[clip=true,trim=60 70 460 700,angle=270]{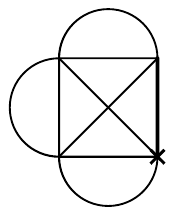}
    \put(17,35){\Set P}
    \put(72,35){\Set P}
    \put(45,65){\Set P}
  \end{overpic}
  \raisebox{16pt}{$\mathaccent\cdot\cup$}%
  \begin{overpic}[clip=true,trim=60 70 460 700,angle=270]{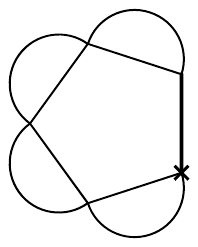}
    \put(68,70){\Set P}
    \put(25,70){\Set P}
    \put(12,30){\Set P}
    \put(78,30){\Set P}
  \end{overpic}
  \raisebox{16pt}{$\mathaccent\cdot\cup$}%
  \begin{overpic}[clip=true,trim=60 70 460 700,angle=270]{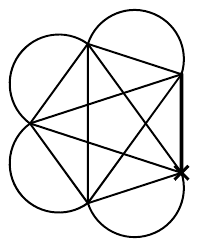}
    \put(68,70){\Set P}
    \put(25,70){\Set P}
    \put(12,30){\Set P}
    \put(78,30){\Set P}
  \end{overpic}
  \raisebox{16pt}{$\mathaccent\cdot\cup\dots$}%
  \caption{The decomposition of the set of Ptolemy diagrams with a
    distinguished base edge.}
  \label{fig:decomposition}
\end{figure}
The set $\Set P$ of Ptolemy diagrams with distinguished base edge can
be described recursively as indicated in
Figure~\ref{fig:decomposition}, see
\cite[Proposition~2.4]{HolmJorgensenRubey2010}.  More precisely,
$\Set P$ is the disjoint union of%
\begin{enumerate}[(i)]
\item the degenerate Ptolemy diagram, consisting of two vertices
  and the distinguished base edge only,
\item a triangle with a distinguished base edge and two Ptolemy
  diagrams glued along their distinguished base edges onto the
  other edges,
\item a clique, {\it i.e.}, a diagram with at least four edges and
  all diagonals present, with a distinguished base edge and Ptolemy
  diagrams glued along their distinguished base edges onto the
  other edges.
\item an empty cell, {\it i.e.}, a polygon with at least four edges
  without diagonals, with a distinguished base edge and Ptolemy
  diagrams glued along their distinguished base edges onto the
  other edges,
\end{enumerate}
In other words, Ptolemy diagrams are obtained by gluing together \lq
elementary\rq\ Ptolemy diagrams, {\it i.e.}, triangles, cliques and
empty cells.  Thus, one could also describe Ptolemy diagrams as
polygon dissections (see, {\it e.g.}  \cite[Proposition~6.2.1 ({\rm
  vi})]{EC2}), in which each region of size at least four receives
one of two colours.

A natural operation on Ptolemy diagrams is (counterclockwise)
rotation.  In the cluster category, rotation corresponds to
Auslander-Reiten translation $\tau$, or, equivalently, application of
the suspension functor $\Sigma$.  Both the total number of Ptolemy
diagrams with distinguished base edge, and the number of diagrams up
to rotation, were already determined, see~\cite[Theorem~B and
Proposition~3.4]{HolmJorgensenRubey2010}.  The central goal of this
article is to determine the number of Ptolemy diagrams that are
invariant under rotation by a given angle.

Apart from rotation, another natural operation on Ptolemy diagrams
consists of replacing every clique by an empty cell and vice versa.
More generally, given \emph{any} set of diagonals $\Diag A$ of the
$(n+3)$-gon, let $\nc \Diag A$ be the set of diagonals that cross no
diagonal in $\Diag A$.  It turns out that $\Diag A$ is a Ptolemy
diagram if and only if $\Diag A=\nc\nc\Diag A$,
see~\cite[Proposition~2.6]{HolmJorgensenRubey2010}.  Note that
precisely those Ptolemy diagrams that are triangulations remain
invariant under this operation.

Let $\Cat A$ be a subcategory of the cluster category of Dynkin type
$A_n$, closed under direct sums and direct summands, that corresponds
to a set of diagonals $\Diag A$.  Then the \Dfn{perpendicular
  subcategory} $\Cat A^\perp$ is the subcategory of the same cluster
category that consists of those objects that have only the zero map
to any object in $\Cat A$:
$$
\Cat A^\perp = \{c: \Hom(c, a)=0\text{ for all } a\in \Cat A\}
$$
One can show that $(\Cat A, \Cat A^\perp)$ is a torsion pair if and
only if $\Diag A$ is a Ptolemy diagram,
see~\cite[Proposition~2.3]{HolmJorgensenRubey2010}.  In this setting,
$\Cat A^\perp$ also corresponds to a Ptolemy diagram, namely
$\Sigma\nc\Diag A$.  As a corollary of our main result we also obtain
the number of Ptolemy diagrams invariant under taking perpendicular
subcategories a given number of times.

\subsection{Cyclic Sieving Phenomena}
The cyclic sieving phenomenon was first described in 2004 by Victor
Reiner, Denis Stanton and Dennis White~\cite{MR2087303}.  It involves
a finite set $\Set X$, a cyclic group $C$ of order $n$ acting on
$\Set X$, and a polynomial $X(q)$.

\begin{dfn}
  The triple $\big(\Set X,C,X(q)\big)$ \Dfn{exhibits the cyclic
    sieving phenomenon} if for every $c\in C$ we have
  $$
  X(\omega_{o(c)})=\size{\Set X^c},
  $$
  where $o(c)$ denotes the order of $c\in C$, $\omega_\di$ is a
  $\di^{th}$ primitive root of unity and $\Set X^c=\left\{x\in \Set
    X:c(x)=x\right\}$ denotes the set of fixed points of $\Set X$
  under the action of $c\in C$.
\end{dfn}
In particular, $X(1)=\size{\Set X}$, {\it i.e.}, $X(q)$ is a
$q$-analogue of the generating function for $\Set X$.

We remark that the cyclic sieving polynomial $X$ is unique only
modulo $q^n-1$.  However, for the unique cyclic sieving polynomial of
degree at most $n-1$, there is an alternative description, which
makes the combinatorics of the orbit structure of $C$ acting on $\Set
X$ explicit.
\begin{prop}[\protect{\cite[Proposition~2.1]{MR2087303}}]
  Let the \Dfn{stabiliser order} of a $C$-orbit of $\Set X$ be the
  number of elements $c\in C$ that fix an element (and therefore all
  elements) in this orbit.

  For $0\leq\ell<n$ let $a_\ell$ be the number of orbits of the
  action of $C$ on $\Set X$, whose stabiliser order divides $\ell$,
  and let $X(q)=\sum_{\ell=0}^{n-1} a_\ell q^\ell$.  Then the triple
  $\big(\Set X,C,X(q)\big)$ exhibits the cyclic sieving phenomenon.
\end{prop}
In particular, $a_0$ is the total number of orbits and $a_1$ is the
number of free orbits.

\subsection{Main Theorems}
\begin{figure}
  \begin{overpic}[width=0.8\textwidth,clip=true,trim=80 608 269
    105]{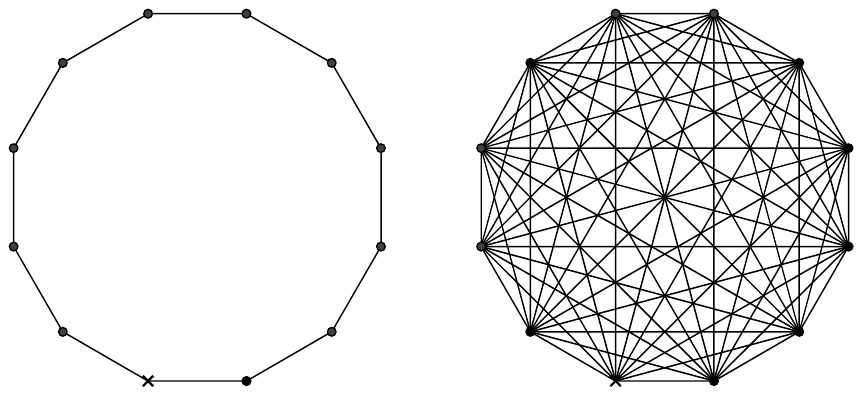}
    \put(13,2){$P^{(4)}_{11, 0, 0, 1}=1$}%
    \put(67,2){$P^{(4)}_{11, 0, 1, 0}=1$}
  \end{overpic}
  \begin{overpic}[width=0.8\textwidth,clip=true,trim=90 649 260 60]{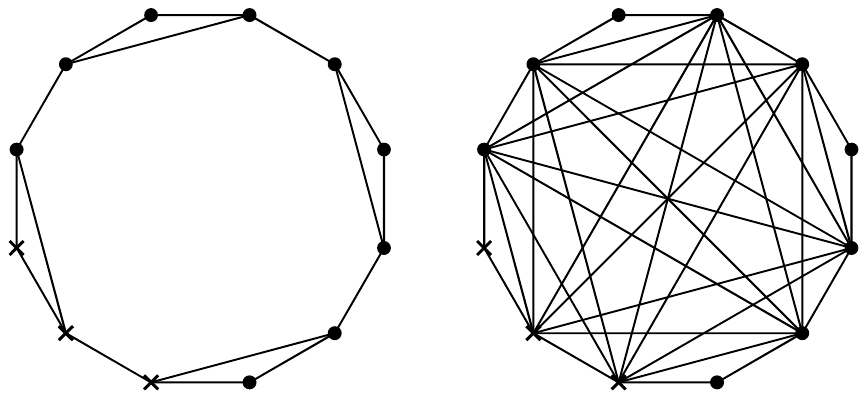}
    \put(13,2){$P^{(4)}_{11, 4, 0, 1}=3$}%
    \put(67,2){$P^{(4)}_{11, 4, 1, 0}=3$}
  \end{overpic}
  \begin{overpic}[height=\textwidth,clip=true,trim=60 80 203 220,angle=270]{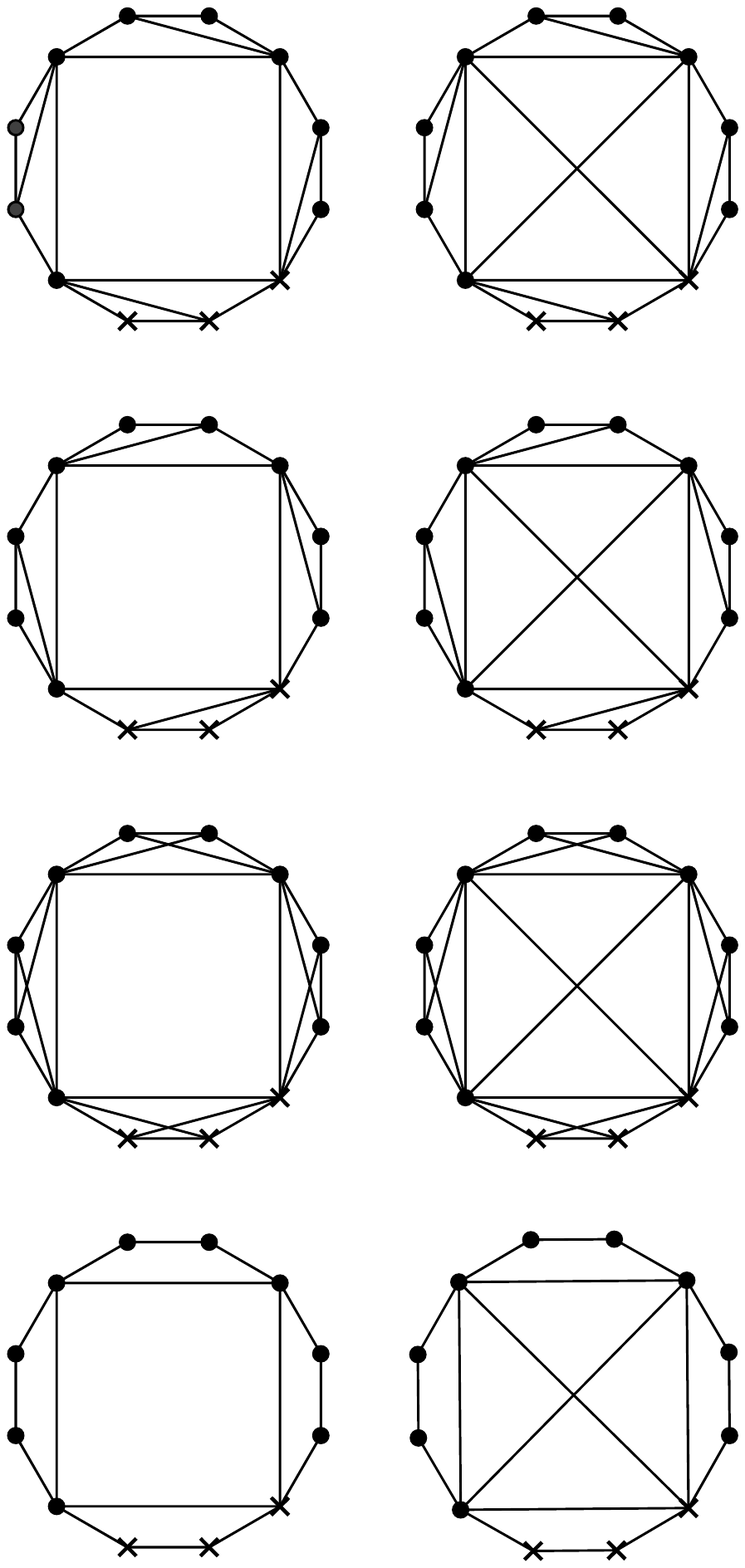}
    \put(9,2){$P^{(4)}_{11, 0, 1, 4}=P^{(4)}_{11, 0, 5, 0}=3$}%
    \put(9,8){$P^{(4)}_{11, 0, 0, 5}=P^{(4)}_{11, 0, 4, 1}=3$}%
    \put(69,2){$P^{(4)}_{11, 8, 1, 0}=6$}%
    \put(69,8){$P^{(4)}_{11, 8, 0, 1}=6$}
  \end{overpic}
  \caption{Ptolemy diagrams on the twelve-gon with fourfold symmetry.
    Crosses indicate different choices of base vertices.}
  \label{fig:rotation}
\end{figure}
\begin{thm}\label{thm:number}
  Let $\Set{P}_{N,\tr,\po,\cl}$ be the set of Ptolemy diagrams on the
  $(N+1)$-gon with a distinguished base edge, with $\tr$ triangles,
  $\po$ cliques of size at least four and $\cl$ empty cells of size
  at least four.  Then the cardinality of $\Set{P}_{N,\tr,\po,\cl}$
  is
  $$
  P_{N,\tr,\po,\cl}=%
  \frac{1}{N}%
  \binom{N-1+\tr+\po+\cl}{N-1,\tr,\po,\cl}%
  \binom{N-2-\tr-\po-\cl}{\po+\cl-1},
  $$
  where we set $\binom{n}{n}=1$ for $n\in\integers$.
\end{thm}

\begin{thm}\label{thm:explicit}
  Let $P^{(\di)}_{N,\tr,\po,\cl}$ be the number of Ptolemy
  diagrams in $\Set{P}_{N,\tr,\po,\cl}$ that are invariant under
  rotation by $2\pi/\di$.  Then, for $\di\geq 2$ a divisor of $N+1$,
  $$
  P^{(\di)}_{N,\tr,\po,\cl}=%
  \binom{\frac{N+1}{\di}-1+\lfloor\frac{\tr+\po+\cl}{\di}\rfloor}%
  {\frac{N+1}{\di}-1, \lfloor\frac{\tr}{\di}\rfloor,
    \lfloor\frac{\po}{\di}\rfloor, \lfloor\frac{\cl}{\di}\rfloor}
  \binom{\lfloor\frac{N-2-\tr-\po-\cl}{\di}\rfloor}{\lfloor\frac{\po+\cl-1}{\di}\rfloor}
  $$
  if $N-2-\tr-\po-\cl\geq \po+\cl-1$ and
  \begin{enumerate}[(i)]
  \item $\di=2$ and $\tr\equiv\po\equiv\cl\equiv 0\pmod\di$, or
  \item $\di=3$ and $\tr\equiv 1\pmod\di$, $\po\equiv\cl\equiv
    0\pmod\di$, or
  \item $\di\geq 2$ arbitrary and $\tr\equiv\po\equiv\cl-1\equiv
    0\pmod\di$ or $\tr\equiv\po-1\equiv\cl\equiv 0\pmod\di$.
  \end{enumerate}
  In all other cases, $P^{(\di)}_{N,\tr,\po,\cl}=0$.
\end{thm}
\begin{rmk}
  Since $P^{(\frac{N+1}{\bi})}_{N,\tr,\po,\cl}=%
  P^{(\frac{N+1}{g})}_{N,\tr,\po,\cl}$, where $g$ is the greatest
  common divisor of $N+1$ and $\bi$, the assumption that $\di$ is an
  integer is not a real restriction.
\end{rmk}
\begin{rmk}
  Expressions for the generating functions of
  $P^{(\di)}_{N,\tr,\po,\cl}$ are given in
  Lemma~\ref{thm:decomposition}.
\end{rmk}
As an illustration, the Ptolemy diagrams on the twelve-gon that are
invariant under rotation by $\pi/2$ are shown in
Figure~\ref{fig:rotation}.

Since $\Diag A^\perp=\Sigma\nc\Diag A=\tau\nc\Diag A$, this theorem
also determines the number of Ptolemy diagrams whose corresponding
subcategory is invariant under taking perpendicular subcategories $b$
times:
\begin{cor}
  Let $P^{\perp^\bi}_{N,\tr,\po,\cl}$ be the number of Ptolemy
  diagrams in $\Set{P}_{N,\tr,\po,\cl}$ invariant under $\bi$-fold
  application of taking perpendiculars.  Let $\di=\frac{N+1}{\bi}$,
  then
  \begin{align*}
    P^{\perp^\bi}_{N,\tr,\po,\cl} &= 2^e
    \binom{\frac{N+1}{2}-1+\frac{\tr}{2}+e}%
    {\frac{N+1}{2}-1, \frac{\tr}{2}, e}
    \binom{\frac{N+1}{2}-2-\frac{\tr}{2}-e}{e-1},%
    \intertext{if $\bi$ is odd, $\di=2$, $\tr\equiv 0\mymod\di$ and
      $\po=\cl=e$, in which case the the central region is
      degenerate, and}%
    P^{\perp^\bi}_{N,\tr,\po,\cl} &= P^{(\di)}_{N,\tr,\po,\cl},%
  \end{align*}
  if $\bi$ is odd, $\di=3$, $\tr\equiv 1\pmod\di$ and $\po=\cl=0$, in
  which case the Ptolemy diagram is a triangulation, or if $\bi$ is
  even.  In all other cases, there are no such Ptolemy diagrams.
\end{cor}
As an illustration, the Ptolemy diagrams on the hexagon that are
invariant under taking perpendiculars three times are shown in
Figure~\ref{fig:perpendiculars}.
\begin{figure}[b]
  \begin{overpic}[height=\textwidth,clip=true,trim=55 56 463 540,angle=270]{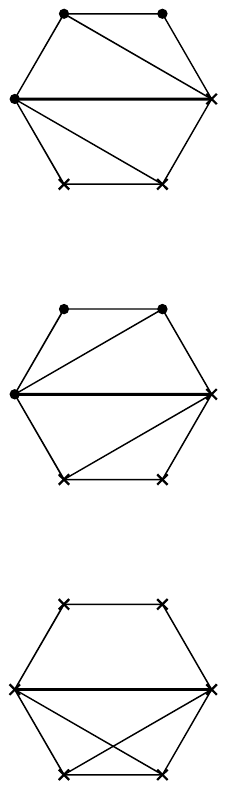}
  \put(25,0){$P^{\perp^3}_{5, 0, 1, 1}=6$}
  \put(78,0){$P^{\perp^3}_{5, 4, 0, 0}=6$}
  \end{overpic}
  \caption{Ptolemy diagrams on the hexagon invariant under taking
    perpendiculars three times.  Crosses indicate different choices
    of base vertices.}
  \label{fig:perpendiculars}
\end{figure}
\begin{proof}
  If $\bi$ is even we have $\Diag A^{\perp^\bi}=(\tau\nc)^\bi\Diag
  A=\tau^\bi\Diag A$, since $\nc$ is an involution on Ptolemy
  diagrams.

  If $\bi$ is odd, $\Diag A^{\perp^\bi}=\tau^\bi\nc\Diag A$.  Thus
  the central region must be degenerate or a triangle.  If it is
  degenerate, we can glue any Ptolemy diagram $\Diag A$ on one side,
  and $\nc\Diag A$ on the other side of the edge, to obtain an
  invariant diagram.  Thus
  \begin{align*}
    P^{\perp^\bi}_{N,\tr, e, e}
    &=\frac{N+1}{2}\sum_{\po=0}^e P_{\frac{N+1}{2},\frac{\tr}{2},\po,e-\po}\\
    &=\sum_{\po=0}^e %
    \binom{\frac{N+1}{2}-1+\frac{\tr}{2}+e}%
    {\frac{N+1}{2}-1, \frac{\tr}{2}, e}%
    \binom{e}{\po}%
    \binom{\frac{N+1}{2}-2-\frac{\tr}{2}-e}{e-1},
  \end{align*}
  as claimed.

  Otherwise, we have to glue three Ptolemy diagrams $\Diag A_1$,
  $\Diag A_2$ and $\Diag A_3$ onto the triangle, one on each side.
  It follows that $\Diag A_1=\nc\Diag A_2=\nc^2\Diag A_3=\nc^3\Diag
  A_1$, and thus that $\Diag A_1=\Diag A_2=\Diag A_3$ is a
  triangulation.
\end{proof}

As another corollary we obtain a (relatively) explicit expression for
the number of Ptolemy diagrams up to rotation.  A (relatively
complicated) expression for the corresponding generating function was
already given in~\cite[Proposition~3.4]{HolmJorgensenRubey2010}.
\begin{cor}
  The number of Ptolemy diagrams in $\Set{P}_{N,\tr,\po,\cl}$ up to
  rotation is
  \begin{equation*}
    \frac{1}{N+1}\sum_{\di|N+1} \phi(\di) P^{(\di)}_{N,\tr,\po,\cl},
  \end{equation*}
  where we set $P^{(1)}_{N,\tr,\po,\cl} = P_{N,\tr,\po,\cl}$.
\end{cor}
Note however that this is not as explicit as it may seem, since
summands vanish depending on the congruence class modulo $\di$.
\begin{proof}
  For every $\di|N+1$ there are $\phi(\di)$ elements of order $\di$
  in the cyclic group of rotations of the $(N+1)-gon$, and for each
  of these rotations there are $P^{(\di)}_{N,\tr,\po,\cl}$ Ptolemy
  diagrams that are invariant.  The corollary now follows from the
  Cauchy-Frobenius formula for the number of orbits
  $$
  \frac{1}{N+1}\sum_{\bi=0}^{N} \#\{\text{Ptolemy diagrams fixed by
    $\tau^\bi$}\}.
  $$
\end{proof}

The remainder of this section is dedicated to a rephrasing of
Theorems~\ref{thm:number} and~\ref{thm:explicit} as a cyclic sieving
phenomenon.
\begin{dfn}
  For $0\le k \le n$ the \Dfn{$q$-binomial coefficient} is
  $$
  \qbinom{n}{k}=\frac{\qi{n}!}{\qi{k}!\qi{n-k}!},
  $$
  where $\qi{n}!=\qi{n}\qi{n-1}\cdots\qi{1}$ and $\qi n =
  1+q\dots+q^{n-1}$.  Analogously, the \Dfn{$q$-multinomial
    coefficient} is
  $$
  \qbinom{n_1+n_2+\dots n_\ell}{n_1, n_2,\dots,
    n_\ell}=\frac{\qi{n_1+n_2+\dots n_\ell}!}{\qi{n_1}!\qi{n_2}!\dots
    \qi{n_\ell}!},
  $$
  where $n_1, n_2,\dots n_\ell$ are non-negative integers.
\end{dfn}

Probably \emph{the} reason why cyclic sieving is called a
\emph{phenomenon}, is that the cyclic sieving polynomial $X(q)$ is
frequently obtained from the ordinary counting function merely by
replacing binomials with $q$-binomials and the like.  This is also
the case in our situation:%
\begin{thm}\label{thm:sieving}
  Let $\Set{P}_{N,\tr,\po,\cl}$ be the set of Ptolemy diagrams on the
  $(N+1)$-gon, with $\tr$ triangles, $\po$ cliques of size at least
  four and $\cl$ empty cells of size at least four.  Let $\tau$ be
  the operation of rotation acting on this set, and let
  $$
  P_{N,\tr,\po,\cl}(q)=\frac{1}{\qi{N}}%
  \qbinom{N-1+\tr+\po+\cl}{N-1,\tr,\po,\cl}%
  \qbinom{N-2-\tr-\po-\cl}{\po+\cl-1},
  $$
  where we set $\qbinom{n}{n}=1$ for $n\in\integers$.

  Then $\big(\Set{P}_{N,\tr,\po,\cl}, \langle\tau\rangle,
  P_{N,\tr,\po,\cl}(q)\big)$ exhibits the cyclic sieving phenomenon.
\end{thm}
\begin{rmk}
  Putting $\po=0$ in the theorem above we obtain polygon dissections
  with a given number of regions.  By contrast, Victor Reiner, Denis
  Stanton and Dennis White~\cite{MR2087303} showed that polygon
  dissections with a given number of diagonals also exhibit the
  cyclic sieving phenomenon.  However, so far we did not manage to
  find a common generalisation, not even for the formula giving the
  total number of dissections of the $n$-gon with $k$ diagonals, {\it
    i.e.}
  $$
  \frac{1}{n+k}\binom{n+k}{k+1}\binom{n-3}{k}
  $$
  and the formula in Theorem~\ref{thm:number}.

  A collection of (partially conjectural) instances of the cyclic
  sieving phenomenon that involve various subsets of (non-crossing)
  diagonals of the $n$-gon was given by Alan Guo in \cite{Guo2010}.
  Two common features of these and also the one described in the
  present article are that the cyclic sieving polynomial is a simple
  product of $q$-binomial coefficients, and to date, no
  representation theoretic proof along the lines of
  \cite[Lemma~2.4]{MR2087303} is available\dots
\end{rmk}

\section{Counting Ptolemy diagrams}
In this section we determine the total number of Ptolemy diagrams as
well as the number of Ptolemy diagrams invariant under rotation by a
given angle.  That is, we provide proofs of Theorem~\ref{thm:number}
and Theorem~\ref{thm:explicit}.

We do so by exhibiting equations for the generating functions, whose
coefficients we then extract using Lagrange inversion:
\begin{thm}[Lagrange inversion]
  Let $F(\yv)$ be a formal power series with $[\yv^0] F(\yv)=0$ and
  $[\yv^1] F(\yv)\neq 0$.  Let $F^{(-1)}(\yv)$ its compositional
  inverse and $H(\yv)$ an arbitrary formal power series.  Then the
  coefficient of $\yv^n$ in $H(F^{(-1)}(\yv))$ is
  $$
  [\yv^n]H(F^{(-1)}(\yv))=\frac{1}{n}[\yv^{n-1}]H'(\yv)\left(\frac{F(\yv)}{\yv}\right)^{-n}.
  $$
\end{thm}
\begin{proof}
  A proof may be found, for example, in~\cite[Corollary 5.4.3]{EC2}.
\end{proof}

\begin{proof}[Proof of Theorem~\ref{thm:number}]

  Let $P_{N,\tr,\po,\cl}$ be the number of Ptolemy diagrams on the
  $(N+1)$-gon with $\tr$ triangles, $\po$ cliques of size at least
  four and $\cl$ empty cells of size at least four.  Then the
  ordinary generating function for Ptolemy diagrams is
  $$
  \Set P(\yv)= \Set P(\yv,\tv,\pv,\cv)=%
  \sum_{N\ge 1,\tr,\po,\cl\ge 0}P_{N,\tr,\po,\cl}\; \yv^N \tv^\tr
  \pv^\po \cv^\cl .
  $$
  Translating the recursive description for the set of Ptolemy
  diagrams given in the introduction into an equation for their
  generating function we obtain
  $$
  \Set P(\yv)= \yv+\tv\Set P(\yv)^2+ (\pv+\cv)\frac{\Set
    P(\yv)^3}{1-P(\yv)},
  $$
  or equivalently,
  $$
  \Set P(\yv)%
  \left(1-\tv\Set P(\yv)%
    -(\pv+\cv)\frac{\Set P(\yv)^2}{1-\Set P(\yv)}\right) = \yv.
  $$
   
  We are now able to apply Lagrange inversion to obtain formulae for
  the coefficients of $\Set P(\yv)$: setting
  $Q(\yv)=\yv\left(1-\tv\yv-(\pv+\cv)\frac{\yv^2}{1-\yv}\right)$ we
  have $\yv=Q(\Set P)$, {\it i.e.} $Q$ is the compositional inverse
  of $\Set P$.  Therefore
  \begin{equation*}
    [\yv^N]\Set P(\yv)%
    =\frac{1}{N}[\yv^{N-1}]\left(\frac{Q(\yv)}{\yv}\right)^{-N}.
  \end{equation*}
  Applying the multinomial theorem 
  $$(1-\tv-\pv-\cv)^{-N}=%
  \sum_{\tr,\po,\cl}\binom{N-1+\tr+\po+\cl}{N-1,\tr,\po,\cl} \tv^\tr
  \pv^\po \cv^\cl
  $$
  we find
  \begin{align*}
    \left(\frac{Q(\yv)}{\yv}\right)^{-N}&=\left(1-\tv\yv-(\pv+\cv)\frac{\yv^2}{1-\yv}\right)^{-N}\\
    &=\sum_{\tr,\po,\cl} \binom{N-1+\tr+\po+\cl}{N-1,\tr,\po,\cl}%
    (\tv\yv)^\tr \pv^\po \cv^\cl \frac{\yv^{2(\po+\cl)}}{(1-\yv)^{\po+\cl}}\\
    &=\sum_{\tr,\po,\cl} \binom{N-1+\tr+\po+\cl}{N-1,\tr,\po,\cl}%
    \tv^\tr \pv^\po \cv^\cl \yv^{\tr+2(\po+\cl)}\\
    &\phantom{=\sum_{\tr,\po,\cl}}\sum_{i}\binom{\po+\cl-1+i}{\po+\cl-1}\yv^i.
  \end{align*}
  Extracting the coefficient of $\yv^{N-1}$ by setting
  $i=N-1-\tr-2(\po+\cl)$ we obtain the desired formula
  \begin{align*}
    P_{N,\tr,\po,\cl}&=[\yv^N\tv^\tr \pv^\po \cv^\cl ]\Set P(\yv,\tv,\pv,\cv)\\
    &=\frac{1}{N}\binom {N-1+\tr+\po+\cl}{N-1,\tr,\po,\cl}
    \binom{N-\tr-\po-\cl-2}{\po+\cl-1}.
  \end{align*}
\end{proof}

\begin{lem}\label{thm:decomposition}
  Let
  $$
  \Set{P}(\yv,\tv,\pv,\cv)=\sum_{N\geq 1,\tr,\po,\cl\geq 0} %
  P_{N,\tr,\po,\cl}\; \yv^N \tv^\tr \pv^\po \cv^\cl
  $$ 
  be the generating function for Ptolemy diagrams.  Then the
  generating function for Ptolemy diagrams that are invariant under
  rotation by $2\pi/\di$ equals
  \begin{align*}
    &\frac{1}{\yv}\yk\,\Set P'(\yk, \tk, \xck, \xdk) %
    \left(1+\Set P(\yk, \tk, \xck, \xdk)%
      \frac{\pv+\cv}{1-\Set P(\yk, \tk, \xck, \xdk)}\right)
    &\text{ for $\di=2$,}\\
    &\frac{1}{\yv}\yk\,\Set P'(\yk, \tk, \xck, \xdk) %
    \left(\tv+\Set P(\yk, \tk, \xck, \xdk)%
      \frac{\pv+\cv}{1-\Set P(\yk, \tk, \xck, \xdk)}\right)%
    &\text{ for $\di=3$,}\\
    &\frac{1}{\yv}\yk\,\Set P'(\yk, \tk, \xck, \xdk) %
    \frac{\pv+\cv} {1-\Set P(\yk, \tk, \xck, \xdk)}%
    &\text{ for $\di\geq 4$.}
  \end{align*}
  In these formulae, we set $\yk=\yv^\di$, $\tk=\tv^\di$,
  $\xck=\pv^\di$ and $\xdk=\cv^\di$, and the derivative is with
  respect to $\yk$.
\end{lem}
\begin{proof}
  Up to this point we always distinguished an edge of the $(N+1)$-gon
  when counting Ptolemy diagrams.  Clearly, this is equivalent to
  marking one of the $N+1$ vertices of the polygon.  For the
  recursive description in the introduction the former seemed more
  natural, but in this proof it will be more convenient to mark a
  vertex, which we will call the \Dfn{distinguished base vertex}
  henceforth.

  For $\di\geq 2$ we can construct a Ptolemy diagram invariant under
  rotation by $2\pi/\di$ as follows: for any multiple $s$ of $\di$,
  we choose a list of $s/\di$ Ptolemy diagrams.  In the first of
  these, we select one vertex other than the distinguished base
  vertex, which will become the distinguished base vertex of the
  diagram we are about to construct.  Then we glue the Ptolemy
  diagrams of $\di$ identical copies of this list in order along
  their distinguished base edges onto the edges of a polygon with $s$
  vertices.  Of course, in the degenerate case $s=2$ we simply have
  two identical Ptolemy diagrams which we glue onto each other along
  their distinguished base edges.  Finally, if $s\geq 4$, we choose
  whether this central region should be an clique or an empty cell.

  \begin{figure}[h]
    \centering
    \raisebox{40pt}{$\bigg($}%
    \begin{overpic}[clip=true,trim=100 110 416 695,angle=270]{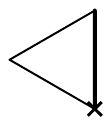}
    \end{overpic}
    \raisebox{34pt}{$,\;$}%
    \begin{overpic}[clip=true,trim=100 110 416 695,angle=270]{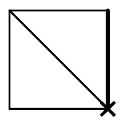}
    \end{overpic}
    \raisebox{34pt}{$,\;$}%
    \begin{overpic}[clip=true,trim=100 110 444 695,angle=270]{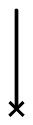}
    \end{overpic}
    \raisebox{40pt}{$\bigg)\mapsto\;$}%
    \begin{overpic}[clip=true,trim=85 110 415 605,angle=270]{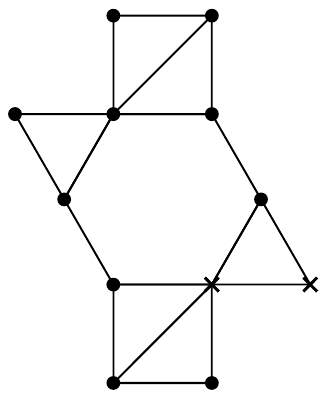}
    \end{overpic}
    \caption{Constructing a Ptolemy diagram invariant under rotation
      by $\pi/4$.  Crosses indicate different choices of base
      vertices.}
    \label{fig:correspondence}
  \end{figure}

  Conversely, given a Ptolemy diagram invariant under rotation by
  $2\pi/\di$, the central region is the region (or possibly the
  diameter) containing the geometric center of the polygon when drawn
  regular and with all diagonals straight.  Let $s$ be the number of
  vertices of this central region.  Cutting out the central region we
  obtain a circular arrangement of smaller Ptolemy diagrams.  We now
  select the first $s/\di$ diagrams in this arrangement, starting
  with the one that contains the distinguished base edge, {\it i.e.},
  the edge that comes just before the distinguished base vertex when
  going clockwise.  An example for this correspondence with $\di=2$
  and $s=6$ can be found in Figure~\ref{fig:correspondence}.

  Let us translate this description into the expressions for the
  generating functions as given in the statement of the lemma.  To
  this end, recall that the generating function of \Dfn{pointed}
  Ptolemy diagrams, {\it i.e.} diagrams with a vertex other than the
  distinguished base vertex selected, equals $\yv\Set
  P'(\yv,\tv,\pv,\cv)$ (see for example \cite[Section~2.1]{BLL}), and
  the generating function of \Dfn{lists} of Ptolemy diagrams is
  $1/\left(1-\Set P(\yv,\tv,\pv,\cv)\right)$.  Since we attach every
  Ptolemy diagram in the list $\di$ times, the number of vertices,
  triangles, etc., in each diagram has to be multiplied by $\di$,
  which is accomplished by replacing $\yv$ by $\yk$, $t$ by $\tk$,
  etc.  Finally, we have to divide by $\yv$, because this variable
  marks the number of vertices \emph{minus one}.
\end{proof}
	
\begin{proof}[Proof of Theorem~\ref{thm:explicit}]
  Following Lemma~\ref{thm:decomposition}, we will treat $\di=2$, $\di=3$
  and $\di\geq 4$ separately.  However, let us first compute the
  expansions of $\yk\Set P'(\yk,\tk,\xck,\xdk)$ and $\yk\frac{\Set
    P'(\yk,\tk,\xck,\xdk)}{1-\Set P(\yk,\tk,\xck,\xdk)}$, which will
  be needed for all three cases.  The first expansion is essentially
  Theorem~\ref{thm:number}:
  \begin{multline*}
    \yk\Set P'(\yk,\tk,\xck,\xdk)%
    =\sum_{n,\tr,\po,\cl}nP_{n,\tr,\po,\cl}\yk^n\tk^\tr \xck^\po \xdk^\cl\\
    =\sum_{n,\tr,\po,\cl}%
    \binom{n-1+\tr+\po+\cl}{n-1,\tr,\po,\cl}
    \binom{n-2-\tr-\po-\cl}{\po+\cl-1}%
    \yk^n\tk^\tr \xck^\po \xdk^\cl.
  \end{multline*}
  For the second, we compute
  \begin{align*}
    [\yk^n] \yk\frac{\Set P'(\yk,\tk,\xck,\xdk)}%
    {1-\Set P(\yk,\tk,\xck,\xdk)}%
    &=[\yk^{n-1}] \frac{\Set P'(\yk,\tk,\xck,\xdk)}%
    {1-\Set P(\yk,\tk,\xck,\xdk)}\\
    &= [\yk^{n-1}] \left(\log\frac{1}%
      {1-\Set P(\yk,\tk,\xck,\xdk)}\right)' \\
    &= n [\yk^n] \log\frac{1}{1-\Set P(\yk,\tk,\xck,\xdk)} \\
    &= [\yk^{n-1}] \frac{1}{1-\yk}%
    \left(\frac{Q(\yk)}{\yk}\right)^{-n}.
  \end{align*}
  In the last line we used Lagrange inversion with
  $H(\yk)=\log\left(1/(1-\yk)\right)$ and
  $Q(\yk)=\yk\left(1-\tk\yk-(\xck+\xdk)\frac{\yk^2}{1-\yk}\right)$.
  The expansion of $\left(\frac{Q(\yk)}{\yk}\right)^{-n}$ was already
  computed in the proof of Theorem~\ref{thm:number}; taking into
  account the additional factor $\frac{1}{1-\yk}$ we obtain
  \begin{equation*}
    \yk\frac{\Set P'(\yk)}{1-\Set P(\yk)}=%
    \sum_{n,\tr,\po,\cl} \binom{n-1+\tr+\po+\cl}{n-1,\tr,\po,\cl}%
    \binom{n-1-\tr-\po-\cl}{\po+\cl} \yk^n\tk^\tr \xck^\po \xdk^\cl .
  \end{equation*}
  
  \textbf{Case $\di\ge 4$.}  By Lemma~\ref{thm:decomposition}, we
  need to compute the coefficient of
  $\yv^N=\yv^{kn-1}=\frac{1}{\yv}\yk^n$ in
  $$
  \frac{1}{\yv}\yk\,\Set P'(\yk,\tk,\xck,\xdk) %
  \frac{\pv+\cv}{1-\Set P(\yk,\tk,\xck,\xdk)},
  $$
  where $\yk=\yv^\di$, $\tk=\tv^\di$, $\xck=\pv^\di$ and
  $\xdk=\cv^\di$.

  Thus, the exponent of $\tv$ and of one of $\pv$ and $\cv$ must be
  divisible by $\di$, while the exponent of the other variable equals
  $1\pmod\di$.  We conclude that the number
  $P^{(\di)}_{N,\tr,\po,\cl}$ of Ptolemy diagrams in $\Set
  P_{N,\tr,\po,\cl}$ that are invariant under rotation by
  $\frac{2\pi}{\di}$ is
  $$
  \binom{\frac{N+1}{\di}-1+\frac{\tr}{\di}+\frac{\po}{\di}+\frac{\cl-1}{\di}}%
  {\frac{N+1}{\di}-1,\frac{\tr}{\di},\frac{\po}{\di},\frac{\cl-1}{\di}}
  \binom{\frac{N+1}{\di}-1-\frac{\tr}{\di}-\frac{\po}{\di}-\frac{\cl-1}{\di}}%
  {\frac{\po}{\di}+\frac{\cl-1}{\di}}
  $$
  if $N+1\equiv \tr\equiv \po\equiv \cl-1\equiv 0\pmod\di$,
  $$
  \binom{\frac{N+1}{\di}-1+\frac{\tr}{\di}+\frac{\po-1}{\di}+\frac{\cl}{\di}}%
  {\frac{N+1}{\di}-1,\frac{\tr}{\di},\frac{\po-1}{\di},\frac{\cl}{\di}}
  \binom{\frac{N+1}{\di}-1-\frac{\tr}{\di}-\frac{\po-1}{\di}-\frac{\cl}{\di}}%
  {\frac{\po-1}{\di}+\frac{\cl}{\di}}
  $$
  if $N+1\equiv \tr\equiv \po-1\equiv \cl\equiv 0\pmod\di$, and $0$
  otherwise, as claimed.
	
  \textbf{Case $\di=2$.}  By Lemma~\ref{thm:decomposition}, we need
  to compute the coefficient of $\yv^N$ in
  \begin{align}\notag
    \frac{1}{\yv}\yk\,\Set P'(\yk, \tk, \xck, &\xdk) %
    \left(1+\Set P(\yk, \tk, \xck, \xdk)%
      \frac{\pv+\cv}%
      {1-\Set P(\yk, \tk, \xck, \xdk)}\right)\\\label{eq:3}%
    =&\frac{1}{\yv}\yk\,\Set P'(\yk, \tk, \xck, \xdk)\\\label{eq:4}
    &+\frac{1}{\yv}(\pv+\cv)\left(\yk%
      \frac{\Set P'(\yk, \tk, \xck, \xdk)}%
      {1-P(\yk, \tk, \xck, \xdk)}%
      -\yk\Set P'(\yk, \tk, \xck, \xdk)\right)
  \end{align}

  Using the expansions for $\yk\Set P'(\yk)$ and $\yk\frac{\Set
    P'(\yk)}{1-\Set P(\yk)}$ and the recurrence
  $\binom{n}{k}=\binom{n-1}{k-1}+\binom{n-1}{k}$ we obtain
  \begin{multline*}
    \yk\frac{\Set P'(\yk)}{1-\Set P(\yk)}-\yk \Set P'(\yk)\\
    \begin{aligned}
    =&\sum_{n,\tr,\po,\cl} \binom{n-1+\tr+\po+\cl}{n-1,\tr,\po,\cl}
    \binom{n-1-\tr-\po-\cl}{\po+\cl}
    \yk^n\tk^\tr \xck^\po \xdk^\cl \\
    &-\sum_{n,\tr,\po,\cl} \binom{n-1+\tr+\po+\cl}{n-1,\tr,\po,\cl}
    \binom{n-2-\tr-\po-\cl}{\po+\cl-1}
    \yk^n\tk^\tr \xck^\po \xdk^\cl \\
    =&\sum_{n,\tr,\po,\cl} \binom{n-1+\tr+\po+\cl}{n-1,\tr,\po,\cl}
    \binom{n-2-\tr-\po-\cl}{\po+\cl} \yk^n\tk^\tr \xck^\po \xdk^\cl .
    \end{aligned}
  \end{multline*}
  
  We can now extract the coefficient of $\yv^N$ separately from the
  summands \eqref{eq:3} and \eqref{eq:4} to derive the expressions
  claimed.

  \textbf{Case $\di=3$.}  By Lemma~\ref{thm:decomposition}, we need to
  compute the coefficient of $y^N$ in
  \begin{align}\notag
    \frac{1}{\yv}\yk\,\Set P'(\yk, \tk, \xck, &\xdk) %
    \left(\tv+\Set P(\yk, \tk, \xck, \xdk)%
      \frac{\pv+\cv}{1-\Set P(\yk, \tk, \xck, \xdk)}\right)\\%
    \label{eq:5}%
    =&\frac{1}{\yv}\tv\yk \Set P'(\yk, \tk, \xck, \xdk)\\%
    \label{eq:6}
    &+\frac{1}{\yv}(\pv+\cv)\left(\yk%
      \frac{\Set P'(\yk, \tk, \xck, \xdk)}%
      {1-P(\yk, \tk, \xck, \xdk)}%
      -\yk\Set P'(\yk, \tk, \xck, \xdk)\right).
  \end{align}
  This is completely analogous to the case $\di=2$, the only difference
  being the factor $\tv$ in the first summand.
\end{proof}

\section{Evaluating the $q$-binomial coefficients}
In this section we show Theorem~\ref{thm:sieving} by evaluating the
expression given there at roots of unity, and thus checking that the
result indeed equals the expression in Theorem~\ref{thm:explicit}.
The evaluations of the $q$-binomial coefficients will be based on the
$q$-Lucas theorem:
\begin{lem}[$q$-Lucas theorem]
  Let $\omega$ be a primitive $\di$\textsuperscript{th} root of unity
  and $a$ and $b$ non-negative integers.  Then
  \begin{equation*}
    \qbinom[\omega]{a}{b}=%
    \binom{\lfloor\frac{a}{\di}\rfloor}{\lfloor\frac{b}{\di}\rfloor}
    \qbinom[\omega]{a-\di\lfloor\frac{a}{\di}\rfloor}{b-\di\lfloor\frac{b}{\di}\rfloor}.
  \end{equation*}
  In particular, if $b\equiv 0\pmod\di$
  \begin{equation*}
    \qbinom[\omega]{a}{b}=%
    \binom{\lfloor\frac{a}{\di}\rfloor}{\lfloor\frac{b}{\di}\rfloor}.
  \end{equation*}
\end{lem}
\begin{proof}
  A proof may be found, for example,
  in~\cite[Theorem~2.2]{MR1176155}.
\end{proof}

For greater clarity, we formulate the evaluation of the
$q$-multinomial coefficient appearing in Theorem~\ref{thm:sieving} as
a lemma:
\begin{lem}
  Let $\omega$ be a primitive $\di$\textsuperscript{th} root of unity
  with $\di\geq 2$ and $n\equiv 0\pmod\di$.  Then
  \begin{align*}
    \frac{1}{\qi[\omega]{n-1}}\qbinom[\omega]{n-2+a+b+c}{n-2,a,b,c}&=%
    \binom{\lfloor\frac{n+a+b+c}{\di}\rfloor-1}%
    {\frac{n}{\di}-1,%
      \lfloor\frac{a}{\di}\rfloor,%
      \lfloor\frac{b}{\di}\rfloor,%
      \lfloor\frac{c}{\di}\rfloor}%
    \intertext{if one of $a$, $b$ and $c$ equals $1\mymod\di$ and the
      others equal $0\mymod\di$, or $\di=2$ and $a\equiv b\equiv
      c\equiv 0\mymod\di$.  Furthermore,}
    \frac{1}{\qi[\omega]{n-1}}\qbinom[\omega]{n-2+a+b+c}{n-2,a,b,c}&=0
  \end{align*}
  otherwise, except if $\di>2$ and $a\equiv b\equiv c\equiv
  0\pmod\di$ -- we do not make a statement for this case.
\end{lem}
\begin{proof}
  Let us first write rewrite the multinomial coefficient as a product
  of binomial coefficients:
  \begin{multline*}
    \frac{1}{\qi{n-1}}\qbinom{n-2+a+b+c}{n-2,a,b,c}\\
    =\frac{1}{\qi{n-1}}\qbinom{n-2+a}{a}\qbinom{n-2+a+b}{b}\qbinom{n-2+a+b+c}{c}.
  \end{multline*}

  Since $\frac{1}{\qi{n-1}}\qbinom{n-2+a+b+c}{n-2,a,b,c}$ is
  symmetric in $a$, $b$ and $c$ it is sufficient to consider the
  following cases to prove the first equality:

  If $a\equiv 1\pmod\di$, the $q$-Lucas theorem and
  $\qi[\omega]{n-1}=\qi[\omega]{\di-1}$ implies
  \begin{align*}
    \frac{1}{\qi[\omega]{n-1}}\qbinom[\omega]{n-2+a}{a}
    &=\frac{1}{\qi[\omega]{\di-1}}%
    \binom{\lfloor\frac{n-2+a}{\di}\rfloor}{\lfloor\frac{a}{\di}\rfloor}%
    \qbinom[\omega]{\di-1}{1}\\
    &=\binom{\lfloor\frac{n-2+a}{\di}\rfloor}{\lfloor\frac{a}{\di}\rfloor}
    =\binom{\lfloor\frac{n+a}{\di}\rfloor-1}{\lfloor\frac{a}{\di}\rfloor}.
  \end{align*}

  If $a\equiv 0\pmod\di$ and $\di=2$ we obtain by similar means
  \begin{align*}
    \frac{1}{\qi[\omega]{n-1}}\qbinom[\omega]{n-2+a}{a}
    &=\frac{1}{\qi[\omega]{1}}%
    \binom{\frac{n-2+a}{\di}}{\frac{a}{\di}}\\%
    &=\binom{\lfloor\frac{n+a}{\di}\rfloor-1}{\lfloor\frac{a}{\di}\rfloor}.
  \end{align*}

  Suppose now that $b\equiv c\equiv 0\pmod\di$ and $a\equiv 0\pmod\di$
  or $a\equiv 1\pmod\di$.  Then, again by the $q$-Lucas theorem,
  $$
  \qbinom[\omega]{n-2+a+b}{b}%
  =\binom{\lfloor\frac{n-2+a+b}{\di}\rfloor}{\frac{b}{\di}}
  =\binom{\lfloor\frac{n+a+b}{\di}\rfloor-1}{\lfloor\frac{b}{\di}\rfloor}
  $$
  and
  $$
  \qbinom[\omega]{n-2+a+b+c}{c}%
  =\binom{\lfloor\frac{n-2+a+b+c}{\di}\rfloor}{\frac{c}{\di}}%
  =\binom{\lfloor\frac{n+a+b+c}{\di}\rfloor-1}{\lfloor\frac{c}{\di}\rfloor}.
  $$

  To show the second equality, again taking advantage of the
  symmetry, we only have to distinguish two cases: on the one hand,
  if $a\equiv b\equiv 1\pmod\di$, we have
  $$
  \qbinom[\omega]{n-2+a+b}{b}=\binom{\frac{n-2+a+b}{\di}}{\frac{b-1}{\di}}
  \qbinom[\omega]{0}{1}=0.
  $$
  On the other hand, if $a\equiv e\pmod\di$ with $e\geq 2$, then
  $$
  \qbinom[\omega]{n-2+a}{a}=%
  \binom{\lfloor\frac{n-2+a}{\di}\rfloor}{\lfloor\frac{a}{\di}\rfloor}
  \qbinom[\omega]{e-2}{e}=0.
  $$
\end{proof}
\begin{proof}[Proof of Theorem~\ref{thm:sieving}]
  Let $\di\geq 2$, $N+1\equiv 0\pmod\di$ and $\omega$ a primitive
  $\di$\textsuperscript{th} root of unity.  We have to show that
  $P_{N,\tr,\po,\cl}(\omega)=P^{(\di)}_{N,\tr,\po,\cl}$.  To this
  end, let
  \begin{align*}
    M_q &
    =\frac{1}{\qi{N}}\qbinom{N-1+\tr+\po+\cl}{N-1,\tr,\po,\cl},%
    & B_q &=\qbinom{N-2-\tr-\po-\cl}{\po+\cl-1},\\
    M &=
    \binom{\frac{N+1}{\di}-1+\lfloor\frac{\tr+\po+\cl}{\di}\rfloor}%
    {\frac{N+1}{\di}-1, \lfloor\frac{\tr}{\di}\rfloor,
      \lfloor\frac{\po}{\di}\rfloor, \lfloor\frac{\cl}{\di}\rfloor},%
    & B &=\binom{\lfloor\frac{N-2-\tr-\po-\cl}{\di}\rfloor}%
    {\lfloor\frac{\po+\cl-1}{\di}\rfloor}.
  \end{align*}

  Let us first check the cases where $M_\omega = M$:
  \begin{enumerate}[(i)]
  \item $\di=2$, $\tr\equiv\po\equiv\cl\equiv 0\pmod\di$:\par
    we have $N-2-\tr-\po-\cl\equiv \po+\cl-1\equiv 1\pmod\di$, thus
    the $q$-Lucas theorem entails $B_\omega = B$.
  \item $\tr\equiv 1\pmod\di$ and $\po\equiv \cl\equiv
    0\pmod\di$:\par
    we have $N-2-\tr-\po-\cl\equiv -4\pmod\di$ and $\po+\cl-1\equiv
    -1\pmod\di$.  Thus,
    \begin{enumerate}[(a)]
    \item if $\di=2$, the $q$-binomial coefficient on the right hand
      side of the $q$-Lucas theorem is $\qbinom[\omega]{0}{1}=0$, and
      thus $B_\omega=0$.
    \item If $\di\geq4$, it is $\qbinom[\omega]{\di-4}{\di-1}=0$, and
      therefore $B_\omega=0$.
    \item However, if $\di=3$, it is
      $\qbinom[\omega]{2\di-4}{\di-1}=\qbinom[\omega]{2}{2}=1$, and
      $B_\omega=B$.
    \end{enumerate}
  \item $\tr\equiv 0\pmod\di$ and $\po\equiv \cl-1\equiv 0\pmod\di$
    or $\po-1\equiv \cl\equiv 0\pmod\di$:\par
    in this case $\po+\cl-1\equiv 0\pmod\di$, and the $q$-Lucas
    theorem entails $B_\omega=B$.
  \end{enumerate}
  It remains to check that $M_\omega\cdot B_\omega = 0$ in all other
  cases.  Suppose that $M_\omega\neq 0$, then $\di>2$ and
  $\tr\equiv\po\equiv\cl\equiv 0\pmod\di$.  Thus
  $N-2-\tr-\po-\cl\equiv -3\pmod\di$ and $\po+\cl-1\equiv
  -1\pmod\di$, and the $q$-binomial coefficient on the right hand
  side of the $q$-Lucas theorem is $\qbinom[\omega]{\di-3}{\di-1}=0$.
\end{proof}

\section*{Acknowledgements}
We would like to thank David Pauksztello for his suggestion to
provide counts for Ptolemy Diagrams invariant under taking
perpendiculars.

\providecommand{\cocoa} {\mbox{\rm C\kern-.13em o\kern-.07em C\kern-.13em
  o\kern-.15em A}}
\providecommand{\bysame}{\leavevmode\hbox to3em{\hrulefill}\thinspace}
\providecommand{\href}[2]{#2}


\begin{thebibliography}{1}

\bibitem{BLL}
Fran{\c{c}}ois Bergeron, Gilbert Labelle, and Pierre Leroux,
  \emph{Combinatorial species and tree-like structures}, Encyclopedia of
  Mathematics and its Applications, vol.~67, Cambridge University Press,
  Cambridge, 1998, Translated from the 1994 French original by Margaret Readdy,
  With a foreword by Gian-Carlo Rota.

\bibitem{Guo2010}
Alan Guo, \emph{Cyclic sieving phenomenon in non-crossing connected graphs},
  Preprint (2010), {math.CO/1007.4584}.

\bibitem{HolmJorgensenRubey2010}
Thorsten Holm, Peter~J\o rgensen, and Martin Rubey, \emph{{Ptolemy diagrams and
  torsion pairs in the cluster category of Dynkin type $A_n$}}, Preprint
  (2010), {math.CO/1010.1184}.

\bibitem{MR2087303}
Victor Reiner, Dennis Stanton, and Dennis White, \emph{The cyclic sieving
  phenomenon}, Journal of Combinatorial Theory, Series A \textbf{108} (2004),
  no.~1, 17--50, {10.1016/j.jcta.2004.04.009}.

\bibitem{MR1176155}
Bruce~E. Sagan, \emph{Congruence properties of {$q$}-analogs}, Advances in
  Mathematics \textbf{95} (1992), no.~1, 127--143,
  {10.1016/0001-8708(92)90046-N}.

\bibitem{EC2}
Richard~P. Stanley, \emph{Enumerative combinatorics. {V}ol. 2}, Cambridge
  Studies in Advanced Mathematics, vol.~62, Cambridge University Press,
  Cambridge, 1999, With a foreword by Gian-Carlo Rota and appendix 1 by Sergey
  Fomin.

\end{thebibliography}
\end{document}